\documentclass[12pt]{amsart}%
\usepackage{graphicx}
\usepackage{amscd}
\usepackage{amsmath}
\usepackage{amsfonts}
\usepackage{amssymb}%
\providecommand{\U}[1]{\protect\rule{.1in}{.1in}}
\def\R{\mathbb R}

\newtheorem{theorem}{Theorem}
\theoremstyle{plain}

\newtheorem{proposition}{Proposition}

\numberwithin{equation}{section}

\begin{document}
\title[The Ptolemaean inequality in $H$--type groups]{The Ptolemaean inequality in $H$--type groups}

\author{A. Fotiadis}
\address{A. Fotiadis\\Department of Mathematics\\ Aristotle University of Thessaloniki\\
Thessaloniki 54.124\\ Greece }
\email{fotiadisanestis@math.auth.gr}
\author{I.D. Platis}
\address{I.D. Platis\\Department of Mathematics\\ University of Crete\\ 70013, Heraklion, Crete \\ Greece }
\email{jplatis@math.uoc.gr} 

\begin{abstract}
We prove the Ptolemaean Inequality and the Theorem of Ptolemaeus
in the setting of $H$--type groups of Iwasawa--type.

\end{abstract}

\maketitle

\section{Introduction}

The purpose of this short note is to give a proof of yet another
generalization of the Ptolemaean Inequality, this time within the
context of a class of Carnot groups. The Ptolemaean Inequality in
planar Euclidean geometry states that given a quadrilateral, then
the product of the euclidean lengths of the diagonals is less or
equal to the sum of the products of the euclidean lengths of its
opposite sides; moreover,  equality holds if and only if the
quadrilateral is inscribed in a circle (Ptolemaeus' Theorem).

It turns out that Ptolemaean Inequality is of intrinsic nature,
and over the past years it has been the subject of  an ongoing
extensive study of many authors who proved its generalization in
various settings. Entirely illustratively, we mention its
generalizations to  normed spaces \cite{S}, CAT(0) spaces
\cite{B-F-W}, M\"{o}bius spaces \cite{B-S}, etc.

In general, we consider in abstract  metric spaces where the
Ptolemaean Inequality holds. To that end, let $(X,d)$ be a metric
space. The metric $d$ is called {\it Ptolemaean} if any four
distinct points $p_1$, $p_2$, $p_3$ and $p_4$ in $X$ satisfy the
Ptolemaean Inequality; that is, for any permutation $(i,j,k,l)$ in
the permutation group $S_4$  we have
\begin{equation}\label{ptolemaean inequality}
d(p_i,p_k)\cdot d(p_j,p_l)\le d(p_i,p_j)\cdot
d(p_k,p_l)+d(p_j,p_k)\cdot d(p_l,p_i).
\end{equation}
In a Ptolemaean space $(X,d)$, we are most interested in the
counterparts of Euclidean circles, that is sets where Ptolemaean
Inequality holds as an equality (Ptolemaeus' Theorem). Explicitly,
a subset $\Sigma$ of $X$ is called a {\it Ptolemaean circle} if
for any four distinct points $p_1, p_2, p_3, p_4 \in \Sigma$ such
that $p_1$ and $p_3$ separate $p_2$ and $p_4$ we have
$$
d(p_1,p_3)\cdot d(p_2,p_4) = d(p_1,p_2)\cdot d(p_3,p_4) +
d(p_2,p_3)\cdot d(p_4,p_1).
$$
In the present article we prove the Ptolemaean Inequality along
with the Theorem of Ptolemaeus in the setting of $H-${\it type
Iwasawa groups $\mathbb{G}$}. These groups form a class of Carnot
groups of step 2 endowed with a metric $d$ (see \ref{metric} and
some additional strong geometric properties, i.e. the existence of
a nice inversion mapping as well as the existence of Ptolemaean
circles which are the $\mathbb{R}-$circles, see Section \ref{H-groups} for further details and
compare also with the properties (I) and (E) in \cite{B-S}. Our
main theorem is the following.
\begin{theorem}
Let $\mathbb{G}$ be an $H-$type  Iwasawa group. Then the metric
$d$ defined in (\ref{metric}) is Ptolemaean and its Ptolemaean
circles are $\mathbb{R}-$circles.
\end{theorem}

The proof of our main result  follows immediately from a simple algebraic relation between  cross--ratios
which are defined in $\mathbb{G}$, see Section \ref{proof}.
Cross--ratios have been used in \cite{P}, to show that generalized
Heisenberg groups with the Kor\'anyi--Cygan metric are Ptolemaean.
It turns out that we can follow the same, but considerably
simplified, line of proof  in the context of $H-$type Iwasawa
groups.

\section{$H-$type  Iwasawa groups}\label{H-groups}

In this section we briefly recall the definition and some basic
properties of $H-$type Iwasawa groups.

Let $\mathfrak{g}$ be a finite--dimensional Lie algebra endowed
with a left invariant inner product $<,>$. Let $\mathfrak{j}$ be
the center of $\mathfrak{g}$ and let also  $\mathfrak{b}$ be the
orthogonal completion of $\mathfrak{j}$ in $\mathfrak{g}$. For
fixed $t\in \mathfrak{j}$ consider the map $J_{t}\colon
\mathfrak{b}\rightarrow \mathfrak{b}$ defined by
$<J_{t}(x),y>=<t,[x,y]>$, where $[\cdot,\cdot]$ is the Lie bracket
of $\mathfrak{g}$. Then, $\mathfrak{g}$ is called an
\emph{$H$--type algebra} if
$[\mathfrak{b},\mathfrak{b}]=\mathfrak{j}$ and moreover $J_{t}$ is
an orthogonal map whenever $<t,t>=1$. An \emph{$H$--type group}
$\mathbb{G}$ is a connected and simply connected Lie group whose
Lie algebra is an $H$--type algebra.

$H$--type groups is a class of step 2 Carnot groups that
generalize the classical Heisenberg group, see for instance \cite[p.12,
p.682]{B-L-U}. More precisely, if $m=\dim{\mathfrak{b}},
n=\dim{\mathfrak{j}}$ and $N=m+n$, then $\mathbb{G}$ is a
homogeneous Carnot group on $\mathbb{R}^{N}$ with dilations
$\delta_{\lambda}(x,t)=(\lambda x, \lambda^{2} t), \; x\in
\mathbb{R}^{m}, \;t\in \mathbb{R}^{n},$ where $\lambda>0$.

Let $\mathbb{G}$ be an $H$--type group. Then, $\mathbb{G}$ is
called an $H$--type \emph{Iwasawa group}  if
for every $x\in \mathfrak{b}$ and for every $t,t^{'}\in
\mathfrak{j}$ with $<t,t^{'}>=0$, there exists $t^{''}\in
\mathfrak{j}$ such that $J_{t}(J_{t^{'}}(x))=J_{t^{''}}(x)$, \cite[p.23]{C-D-K-R}. For
example, any Heisenberg group $\mathbb{H}^{N}$ is an $H$--type
Iwasawa group, \cite[p.702]{B-L-U}.

From now on $\mathbb{G}$ shall always denote an $H$--type Iwasawa
group. Since the exponential mapping is a bijection of
$\mathfrak{g}$ onto $\mathbb{G}$, we shall parametrize $p$ in
$\mathbb{G}$ by $(x, t)\in
\mathfrak{b}\oplus\mathfrak{j}=\mathfrak{g}$, where
$p=\exp{(x,t)}$.  Multiplication in $\mathbb{G}$ is of a special
form, see \cite[p.687]{B-L-U}: there exist $m\times m$ skew
symmetric and orthogonal matrices $U^{1},\ldots,U^{n}$ such that
\begin{eqnarray*}
(x,t)(x^{'},t^{'})&=&\left(x+x', t+t'+\frac{1}{2}[x,x']\right)\\
&=&\left(x+x^{'},t^{1}+{t^{'}}^{1}+<U^{1}x,x^{'}>,\dots,t^{n}+{t^{'}}^{n}+<U^{n}x,x^{'}>\right),
\end{eqnarray*}
for all $(x,t),(x^{'},t^{'})\in \mathbb{G}$. For example, in the
case of the classical Heisenberg group $\mathbb{H}^{1}$ we have
$m=2,n=1$ and
$U^{1}=\left(%
\begin{array}{cc}
  0 & 1 \\
  -1 & 0 \\
\end{array}%
\right)$.

We  note that $(x,t)^{-1}=(-x,-t)$ and also that the matrices
$U^{1},\ldots,U^{n}$ have the following property:
$$
U^{i}U^{j}+U^{j}U^{i}=0,\text{ for every }i,j\leq n \text{ with }
i\neq j.
$$
The  distance $d$ in $\mathbb{G}$ is defined via a gauge
${\widetilde d}$ in a manner similar to the Heisenberg group case.
If $p\in\mathbb{G}$ is parametrized
by $(x,t)\in\mathbb{R}^n\times\mathbb{R}^n$ we set
\begin{equation}
\widetilde{d}(p)=\left( |x|^{4}+16|t|^{2}\right)^{\frac{1}{4}}.
\end{equation}
For our purposes and for clarity as well, we shall prove
\begin{proposition}\label{prop-triangle}
For every $p,q\in\mathbb{G}$ we have
\begin{equation}\label{ineq-triangle}
{\widetilde d}(p^{-1}q)\le {\widetilde d}(p)+{\widetilde d}(q).
\end{equation}
Equality holds if and only if $p=\exp(x,0)$, $q=\exp(\lambda x,0)$ for some $\lambda\in\R$.
\end{proposition}
\begin{proof}
We parametrize   $p^{-1}$ by $(x,t)$ and $q$ by $(x',t)$; then,
$$
{\widetilde
d}(p^{-1}q)^4=|x+x'|^4+16\left|t+t'+\frac{1}{2}[x,x']\right|^2,
$$
which  is equal to
\begin{eqnarray*}
{\widetilde d}(p)^4+{\widetilde d}(q)^4&+&4\left(<x,x'>^2+|[x,x']|^2\right)\\
&+&4\left(|x|^2<x,x'>+4<t,[x,x']>\right)\\
&+&4\left(|x'|^2<x,x'>+4<t',[x,x']>\right)\\
&+& 2\left(|x|^2|x'|^2+16<t,t'>\right).
\end{eqnarray*}
The following inequalities hold.
\begin{eqnarray}
&&\label{ineq1}
\left(|x|^2|x'|^2+16<t,t'>\right)\le {\widetilde d}(p)^2{\widetilde d}(q)^2,\\
&&\label{ineq2}
\left(|x|^2<x,x'>+4<t,[x,x']>\right)\le {\widetilde d}(p)^2\left(<x,x'>^2+|[x,x']|^2\right)^{1/2},\\
&&\label{ineq3}
\left(|x'|^2<x,x'>+4<t',[x,x']>\right)\le {\widetilde d}(q)^2\left(<x,x'>^2+|[x,x']|^2\right)^{1/2},\\
&&\label{ineq4} \left(<x,x'>^2+|[x,x']|^2\right)\le
|x|^2|x'|^2\leq {\widetilde d}(p)^2 {\widetilde d}(q)^2.
\end{eqnarray}
The first three inequalities are immediate. The  last inequality
follows by applying the Cauchy--Schwarz inequality to Cygan's
hermitian form \cite[p.70]{C}
$$
h(x,x')=<x,x'>+i[x,x'].
$$
Combining the above we have:
\begin{equation}\begin{split}
{\widetilde d}(p^{-1}q)^4\le & {\widetilde d}(p)^4+{\widetilde
d}(q)^4 +6{\widetilde d}(p)^2 {\widetilde d}(q)^2+4 {\widetilde
d}(p)^3 {\widetilde d}(q)+4 {\widetilde d}(p) {\widetilde
d}(q)^3\\
=&\left( {\widetilde d}(p)+{\widetilde d}(q)\right)^{4},
\end{split}
\end{equation}
thus we obtain the desired triangle inequality.

Observe now that equality in (\ref{ineq-triangle}) holds if and
only if (\ref{ineq1}), (\ref{ineq2}), (\ref{ineq3}) and
(\ref{ineq4})  hold simultaneously as equalities. Therefore, from
the last inequality considered as an equality we have $x'=\lambda
x$. Hence all other inequalities hold as equalities if and only if
$t=t'=0$.

\end{proof}

The distance $d$ in $\mathbb{G}$ is  defined by
\begin{equation}\label{metric}
d(p,q)=\widetilde{d}(p^{-1}q).
\end{equation}
It is completely obvious that $d(p,q)=0$ if and only if $p=q$,
$d(p,q)=d(q,p)$ and also that $d$ is invariant by left
translations and is scaled up to a factor $\lambda$ when we apply
a dilation $\delta_\lambda$, \cite[p.705]{B-L-U}).
From (\ref{ineq-triangle}) we have that
\begin{equation}\label{ineq-triangle2}
{ d}(p,q)\le {d}(p,0)+{ d}(0,q).
\end{equation}
Thus, if $p,r,q\in\mathbb{G}$ then by invariance we also have
$$
d(p,q)\le d(p,r)+d(r,q).
$$
i.e. the triangular inequality for the distance function $d$.

It is worth remarking at thei point that the definition of the
gauge function $\widetilde{d}$ generalizes to arbitrary Carnot
groups but the corresponding function $d=d(\cdot,\cdot)$ is not in
general a distance but only a pseudo--distance,
\cite[p.300]{D-G-N}, \cite[p.231]{B-L-U}.

The key feature of the class of $H-$type Iwasawa groups is the
existence of a natural inversion, which generalizes the inversion
$\sigma(x)=-\frac{x}{|x|^{2}}$ of $\mathbb{R}^{N}\setminus{\{0\}}$
to arbitrary $H$--type groups. It plays an important role in the
proof   and we shall discuss some of its properties.

The inversion map $\sigma\colon
\mathbb{G}\setminus{\{0\}}\rightarrow \mathbb{G}\setminus{\{0\}}$
is defined by (cf. \cite[p.705]{B-L-U})
$$
\sigma(x,t)=\left(-\frac{|x|^{2}x-4\sum_{k=1}^{n}U^{k}x}{|x|^{4}+16|t|^{2}},-\frac{t}{|x|^{4}+16|t|^{2}}
\right),
$$
and satisfies $\sigma^2=id$.

We shall also consider the one point compactification
$\widehat{\mathbb{G}}$ of $\mathbb{G}$ by adding the point
$\infty$ at infinity. The distance $d$ is extended in
$\widehat{\mathbb{G}}$ in the obvious way and is denoted by the
same symbol: for every $p\in\mathbb{G}$,
$$
d(p,\infty)=+\infty,\quad d(\infty,\infty)=0.
$$
The actions of left translations and dilations are also extended
naturally in $\widehat{\mathbb{G}}$: left translation of any
element of $\widehat{\mathbb{G}}$ by $\infty$ maps it to $\infty$
and  the image $\infty$ by any $\lambda-$dilation is again
$\infty$. For the inversion $\sigma$  we set $\sigma(0)=\infty$
and $\sigma(\infty)=0$; the following holds \cite[p.706]{B-L-U}:
$$
d(\sigma(p),0)=\frac{1}{d(p,0)},\quad
d(\sigma(p),\sigma(q))=\frac{d(p,q)}{d(p,0)\;d(0,q)}.
$$
Note that the last equality holds in an $H$--type group
$\mathbb{G}$ if and only if $\mathbb{G}$ is Iwasawa
\cite[p.23]{C-D-K-R}.

Finally, for every $x\in\mathbb{R}^m\setminus\{0\}$ we define the
{\it standard $\mathbb{R}$--circle} $R_x$ passing through $0$ and
$\infty$ as the set
\[
R_x=\{(\lambda x,0):  \lambda \in \mathbb{R}\}.
\]
An {\it $\mathbb{R}$--circle} is the image of some $R_{x}$ under the
action of the {\it similarity} group $\widehat{\mathbb{G}}$: this
group comprises of maps which are composites of left translations, dilations and the
inversion $\sigma$ (see also \cite[p.9]{C-D-K-R}).


\section{Proof of Theorem 1}\label{proof}

Given four distinct  points $p_{1},p_{2},p_{3},p_{4}\in
\mathbb{G}$ we define their {\it cross--ratio} $\mathbb{X}(p_{1},p_{2},p_{3},$ $p_{4})$ by
\begin{equation}\label{cross ratio}
\mathbb{X}^{1/2}(p_{1},p_{2},p_{3},p_{4})=\frac{d(p_{1},p_{3})d(p_{2},p_{4})}{d(p_{1},p_{4})d(p_{2},p_{3})},
\end{equation}
and the definition is extended if one of the points is $\infty$ in
the obvious way. From the properties of left translations,
dilations and inversion stated in the previous paragraph, one may verify straightforwardly that the cross--ratio
$\mathbb{X}(p_{1},p_{2},p_{3},p_{4})$ is
invariant under the action of the similarity group of
$\widehat{\mathbb{G}}$. This allows us to normalize so that two of
the points of a given quadruple are $0$ and $\infty$.

Given a quadruple $p_1,p_2,p_3,p_4$ of distinct points in
$\widehat{\mathbb{G}}$, let $p_i,p_i,p_k,p_l$ a permutation of
these points. Let also
$$
\mathbb{X}_1=\mathbb{X}(p_i,p_j,p_k,p_l),\quad \mathbb{X}_2=\mathbb{X}(p_i,p_k,p_j,p_l).
$$
From (\ref{ptolemaean inequality}) and (\ref{cross ratio}) it
follows that the Ptolemaean inequality is satisfied  if and only
if
\begin{equation}\label{equiv ptolemaean}
\mathbb{X}_1^{1/2}+\mathbb{X}_2^{1/2}\geq 1.
\end{equation}
 We now normalise so that
$ p_i=\infty,\; p_l=0 $. Thus, we have
\[
\mathbb{X}_{1}^{1/2}=\frac{d(p_{j},0)}{d(p_{j},p_k)},\quad \text{ and }\quad
\mathbb{X}_{2}^{1/2}=\frac{d(p_{k},0)}{d(p_{j},p_k)}.
\]
Consequently, (\ref{equiv ptolemaean}) is equivalent to
\begin{equation}\label{triangular}
d(p_{j},0)+d(0,p_{k})\geq d(p_{j},p_k),
\end{equation} i.e. the triangular
inequality for the distance function $d$. Since the permutation
was arbitrary, this proves the Ptolemaean Inequality
(\ref{ptolemaean inequality}).

Moreover, if equality is valid (Ptolemaeus' Theorem), then
\begin{equation}
\mathbb{X}_1^{1/2}+\mathbb{X}_2^{1/2}= 1,
\end{equation}
which is equivalent to
\begin{equation}
d(p_{j},0)+d(0,p_{k})= d(p_{j},p_k).
\end{equation}

Thus, as we have seen in Section 2, this happens if and only if 
$p_{j}=(x,0)$ and $ p_{k}=(\lambda x,0)$ for some $lambda\in\R$. Since $p_{i}=\infty$ and
$p_{l}=0$ we conclude that $p_i, p_j, p_k, p_l$ are in a standard
$\mathbb{R}$--circle and the proof of the result is complete.

\vspace{10 mm}

Acknowledgements: The first author would like to thank Michel
Marias for his constant support and stimulating discussions.

\end{document}